\documentclass[10pt]{article}
\usepackage{amstext, amsmath, amsthm, amssymb}
\usepackage{amsmath}
\usepackage{geometry} 
\usepackage{color}

\numberwithin{equation}{section}
\usepackage[nottoc]{tocbibind}
\usepackage{hyperref}
\usepackage{makeidx}
\makeindex

\newtheorem{thm}{Theorem}[section]
\newtheorem{cor}[thm]{Corollary}
\newtheorem{lem}[thm]{Lemma}
\newtheorem{prop}[thm]{Proposition}
\newtheorem{defin}[thm]{Definition}
\newtheorem{rem}[thm]{Remark}
\newtheorem{example}[thm]{Example}

\usepackage{colortbl}

\newcommand{\gl}{{\mathfrak{gl}}}

\def\suml{\sum\limits}

\def\dim{\operatorname{dim}}

\def\ad{\mathop{ad}}

\def\sgn{\mathop{sgn}}
\def\tr{\mathop{tr}}

\title{Orthogonalization and polarization of Yangians
}

\author{Wolfgang Bock, Vyacheslav Futorny, Mikhail Neklyudov, Jian Zhang}

\AtEndDocument{\bigskip{\footnotesize%
  (W.\,Bock) \textsc{Linnaeus University, Department of Mathematics, Universitetsplatsen 1, 35252, V\"{a}xj\"{o}, Sweden} \par
  \textit{E-mail address}: \texttt{wolfgang.bock@lnu.se} \par
  \addvspace{\medskipamount}
 (V.\,Futorny) \textsc{Shenzhen International Center for Mathematics, Southern University of Science and Technology, Shenzhen, China} \par
  \textit{E-mail address}: \texttt{vfutorny@gmail.com}  \par
   \addvspace{\medskipamount}
  (M.\,Neklyudov) \textsc{Instituto de Ci\^{e}ncias Exatas, Departamento de Matematica, UFAM, Manaus, CEP 69077-000, Brasil} \par
  \textit{E-mail address}: \texttt{misha.neklyudov@gmail.com}  \par
   \addvspace{\medskipamount}
  (J.\,Zhang) \textsc{School of Mathematics and Statistics, Central China Normal University, Wuhan, Hubei 430079, China} \par
  \textit{E-mail address}: \texttt{jzhang@ccnu.edu.cn}
}}

\date{}

\begin{document}
\maketitle

\begin{abstract}
For every family of orthogonal
polynomials,  we define a new realization of the Yangian of $\gl_n$. Except in the case of Dickson polynomials, the new realizations do not satisfy the RTT relation. We obtain an analogue of the Christoffel-Darboux  formula. Similar construction can be made for any family of functions satisfying certain recurrence relations, for example, $q$-Pochhhammer symbols and Bessel functions. Furthermore, using an analogue of the Jordan-Schwinger map, we define the ternary Yangian for a Lie algebra as a flat deformation of the current algebra of certain ternary extension of the given Lie algebra.

\end{abstract}

\section{Introduction}

Yangians were introduced by Drinfeld in 1985 in his seminal paper \cite{D}   following their exploration in the context of the inverse scattering method by Faddeev's  school \cite{STF,TTF}. Yangians constitute a family of quantum groups associated with rational solutions of the classical Yang–Baxter equation. For any simple finite-dimensional Lie algebra $\mathcal A$ over the field complex numbers, the Yangian $Y(\mathcal A)$ corresponds to a canonical deformation of the universal enveloping algebra $U(\mathcal A[x])$ for the polynomial current Lie algebra $\mathcal A[x]$.

 In this paper, we introduce two constructions of families of algebras generalizing the Yangian of $\gl_n$. The first procedure can be deduced from the elementary identity-motivating definition of the Yangian as follows.  

Let $E:=\left(E_{ij}\right)_{i,j=1}^n$   be the matrix whose entries are the standard matrix units
$\{E_{ij}\}_{i,j=1}^n$ of $Mat(n)$. Then we have following well known identity:
\begin{equation}\label{eqn:Yangian_0}
[(E^{r+1})_{ij},(E^s)_{kl}]-[(E^r)_{ij},(E^{s+1})_{kl}]=(E^r)_{kj}(E^s)_{il}-(E^s)_{kj}(E^r)_{il},i,j=1,\ldots,n,r,s\in \mathbb{N}\cup\{0\}.
\end{equation}
Notice that by linearity we can extend it as follows
\begin{equation}\label{eqn:GenYangian_0}
[((xf)(E))_{ij},(g(E))_{kl}]-[(f(E))_{ij},((xg)(E))_{kl}]=(f(E))_{kj}(g(E))_{il}-(g(E))_{kj}(f(E))_{il},
\end{equation}
($i,j,k,l=1,\ldots,n$) for arbitrary polynomials $f,g\in \mathbb C[x]$. Furthermore, the identity \eqref{eqn:GenYangian_0} is informally valid for arbitrary power series $f$ and $g$ under some natural assumptions. Choosing different sequences of functions $\{f_n\}_{n\in \mathbb{N}}$  satisfying recurrence relations (i.e. $q$-Pochhhammer symbols, orthogonal polynomials etc.) we construct from identity \eqref{eqn:GenYangian_0} new realizations of Yangian $Y(\gl_n)$ or similar algebras. They in general do not satisfy the RTT relation, a notable exception  is the case of Dickson polynomials (\cite{LMT}) considered in Section \ref{sec:Chebyshev} , where this method leads to, as long as authors aware, to the new rational solution of inhomogeneous Yang-Baxter equation.
The construction connects the theory of Yangians with the classical theory of orthogonal polynomials. As a result, we deduce an analogue of the Christoffel-Darboux formula based on the new realization of the Yangian. 

The second construction presented in Section \ref{sec:Polarization} is an analogue of the Jordan-Schwinger map (\cite{BFN2021}) for the Yangian. In the case of the Jordan-Schwinger map, we start with a basis of $\gl_n$ (where $n$ could be finite or infinite), and 
 then construct a realization of an arbitrary Lie algebra. Here,  instead of the basis of $\gl_n$ we start with a basis of the Yangian $Y(\gl_n)$.  For an arbitrary Lie algebra $\mathcal{A}$ of dimension $n$ we linearly ¨deform¨ Yangian $Y(\gl_n)$ into the algebra $Yt(\mathcal{A})$. The surprising fact is that $Yt(\mathcal{A})$ is a flat deformation of the universal enveloping algebra $U(\mathcal{A}^{tern}[x])$ (see Theorem \ref{thm:deformation}), where $\mathcal{A}^{tern}$ is a certain ternary extension of the initial Lie algebra $\mathcal{A}$.
In Example \ref{ex:Ternary} we compute the extension for some low-dimensional algebras which are not necessarily simple. Thus, the algebra $Yt(\mathcal{A})$ can be viewed as Yangian $Y(\mathcal{A}^{tern})$, where  $\mathcal{A}^{tern}$ is not necessarily semisimple. Note that $Yt(\mathcal{A})$ is not a Hopf algebra (unless $\mathcal{A}^{tern}=\gl_m$, for some $m\geq n$), but nevertheless we still have a notion of the evaluation homomorphism (Proposition \ref{prop:Evaluation}). 

\section{New realization of Yangian of $\gl_n$}\label{sec:OYangians}
Recall that the Yangian $Y(\gl_n)$ of $\gl_n$ is a unital algebra generated by  
  ${t}_{ij}^{(r)}$, $1\leq i,j\leq n$, $r\in \mathbb N$
with defining relations given by
\begin{equation} \label{eqn:Yangian_p}
[t^{(r+1)}_{ij},t^{(s)}_{kl}]-[t^{(r)}_{ij},t^{(s+1)}_{kl}]=t^{(r)}_{kj}t^{(s)}_{il}-t^{(s)}_{kj}t^{(r)}_{il},
\end{equation}
where $r,s\in \mathbb N$ , and 
$t^{(0)}_{ij}=\delta_{ij}$.

We define a family of algebras $OY(\gl_n,a,b)$
for arbitrary complex parameters $a=\{a_m\}_{m=0}^{\infty}$, $b=\{b_m\}_{m=0}^{\infty}$ as follows.

\begin{defin}
Let $a=\{a_m\}_{m=0}^{\infty}, b=\{b_m\}_{m=0}^{\infty}$ be sequences of complex parameters and  $OY(\gl_n,a,b)$ be a unital associative algebra over $\mathbb{C}$ with generators
\[
\tilde{t}_{ij}^{(r)}, 1\leq i,j\leq n, r\in \mathbb{N}
\]
 and defining relations given by
\begin{equation}\label{eqn:GenYangian_1}
[\tilde{t}^{(r+1)}_{ij}+ a_r \tilde{t}^{(r)}_{ij}+b_r \tilde{t}^{(r-1)}_{ij},\tilde{t}^{(s)}_{kl}]-[\tilde{t}^{(r)}_{ij},\tilde{t}^{(s+1)}_{kl}+a_s \tilde{t}^{(s)}_{kl}+b_s \tilde{t}^{(s-1)}_{kl}]
=\tilde{t}^{(r)}_{kj}\tilde{t}^{(s)}_{il}-\tilde{t}^{(s)}_{kj}\tilde{t}^{(r)}_{il},
\end{equation}
and
\[
\tilde{t}_{ij}^{(0)}=\delta_{ij},\tilde{t}_{ij}^{(-1)}:=0,
\]
where $i,j=1,\ldots,n,r,s\in \mathbb{N}\cup\{0\}$.
\end{defin}
\begin{rem}
Classical  Yangian relations correspond to the case $a_n=b_n=0$ for all $n$.
\end{rem}
\begin{paragraph}{Generating function.}
Consider the following generating function:
\[
\tilde{T}_{ij}(u)=\delta_{ij}+\sum\limits_{r=1}^{\infty}\tilde{t}^{(r)}_{ij} u^{-r}
\]
For a sequence $\{c_n\}_{n=0}^{\infty}$ we define 
the shifted sequence 
\[
s(c)=\{c_{n+1}\}_{n=0}^{\infty}
\]
and the (pseudo) differential operator $L_z^c$ by
\[
L^c_z:\suml_{k=0}^{\infty} d_k z^{-k}\mapsto \suml_{k=0}^{\infty} c_k d_k z^{-k}.
\]

Now we can rewrite \eqref{eqn:GenYangian_1} as follows
\begin{equation}\label{eqn:GenYangian_2}
\left(u+L^a_u+\frac{1}{u}L^{s(b)}_u-v-L^a_v-\frac{1}{v}L^{s(b)}_v\right)[\tilde{T}_{ij}(u),\tilde{T}_{kl}(v)]=\tilde{T}_{kj}(u)\tilde{T}_{il}(v)-\tilde{T}_{kj}(v)\tilde{T}_{il}(u),
\end{equation}
where $i,j,k,l=1,\ldots,n$.




\end{paragraph}
\begin{paragraph}{Realization.}
Let $\{p_m\}_{m=0}^{\infty}$, be a system of monic polynomials satisfying the recurrence relation
\begin{equation}\label{eqn:Recurrence}
p_{m+1}(x)=(x-a_m)p_m(x)-b_m p_{m-1}(x), m\in \mathbb{N},
\end{equation}
for some sequences $\{a_m,b_m\}_{m=0}^{\infty}$. From now on we will also put
\begin{equation}\label{eqn:abInitCondition}
b_0:=0.
\end{equation}
So the recurrence relation \eqref{eqn:Recurrence} formally holds for $m=0$. 

Then applying formula \eqref{eqn:GenYangian_0} with $f=p_r, g=p_s$, $r,s\in \mathbb{N}$ and denoting
\[
\tilde{t}_{ij}^{(r)}:=(p_r(E))_{ij},
\]
we get the system \eqref{eqn:GenYangian_1}.

In the following we will denote $p_l^m$, $l=0, \ldots, m-1$ the coefficients of polynomial $p_m$ i.e.
\[
p_m=x^m+\sum\limits_{l=0}^{m-1}p_l^m x^l,m\geq 1.
\]
From now on, unless otherwise explicitly stated, we will assume that $b_n\neq 0, n\in \mathbb{N}$, and, consequently, the polynomials $\{p_m\}_{m=0}^{\infty}$ are orthogonal (with respect to some measure $\mu$).
\begin{rem} For each integer $m$ denote by $p_l^{(m)}$ the $m$-th associated  polynomial of the orthogonal polynomial $p_l$. The
transfer from  system\eqref{eqn:GenYangian_1}  of orthogonal polynomials $\{p_l\}_{l=0}^{\infty}$ to the $m$-th associated orthogonal polynomials $\{p_l^{(m)}\}_{l=0}^{\infty}$ corresponds to the transformation $a\mapsto s^{m}(a), b\mapsto s^{m}(b)$ (i.e. shift by $m$ forward) of coefficients of the recurrence relation. Consequently, for the $m$-th associated orthogonal polynomials we get the system
\begin{equation}\label{eqn:GenYangian_2Ass}
\left(u+L^{s^m(a)}_u+\frac{1}{u}L^{s^{m+1}(b)}_u-v-L^{s^m(a)}_v-\frac{1}{v}L^{s^{m+1}(b)}_v\right)[\tilde{T}_{ij}(u),\tilde{T}_{kl}(v)]=\tilde{T}_{kj}(u)\tilde{T}_{il}(v)-\tilde{T}_{kj}(v)\tilde{T}_{il}(u),
\end{equation}
with $i,j,k,l=1,\ldots,n$.
\end{rem}
\begin{defin}
    Define
    \[
    K(z,u):=\sum\limits_{l=0}^{\infty}p_l(z)u^l,
    \]
   the generating function of the sequence of orthogonal monic polynomials.
\end{defin}
Notice that the recurrence relation \eqref{eqn:Recurrence} can be expressed in the following form:
\begin{lem}
\begin{equation}\label{eqn:GenFunctionPol_1}
    \left(u+L^a_u+\frac{1}{u}L^{s(b)}_u\right) K(z,\frac{1}{u})=z K(z,\frac{1}{u})+u 
\end{equation}
\end{lem}
\begin{proof}
    Immediately follows from the recurrence relation \eqref{eqn:Recurrence} and definition of the generating function.
\end{proof}

We will also denote the algebra $OY(\gl_n,a,b)$ by $OY(\gl_n,\{p_m\})$ 
\begin{defin}
  Let us define the map $W:OY(\gl_n,\{p_m\})\to Y(\gl_n)$ as follows
  \[
\tilde{t}_{ij}^{(r)}\to t_{ij}^{(r)}+\sum\limits_{l=0}^{r-1} p_l^r t_{ij}^{(l)},r\geq 1.
  \]
\end{defin}

\begin{thm}\label{thm:ImageYangian}
    The map $W$ is an algebraic isomorphism between $OY(\gl_n,\{p_m\})$ and $Y(\gl_n)$ given by the following formula on generators of $OY(\gl_n,\{p_m\})$:
\begin{equation}\label{eqn:W_IntegralForm}
W(\tilde{t})_{ij}(u):=\frac{1}{2\pi i}\int\limits_{|z|=1} K(z,\frac{1}{u})t_{ij}(z)\frac{dz}{z},i,j=1\ldots,n. 
\end{equation} 
\end{thm}
\begin{proof}
Formula \eqref{eqn:W_IntegralForm} follows from the definition of the generating function $K$ and the formal expansion of the RHS of \eqref{eqn:W_IntegralForm} as a power series in $u^{-1}$. Indeed, we have
\begin{eqnarray}
    \frac{1}{2\pi i}\int\limits_{|z|=1} K(z,\frac{1}{u})&t_{ij}(z)\frac{dz}{z}=\frac{1}{2\pi i}\sum\limits_{l=0}^{\infty}u^{-l}\int\limits_{|z|=1}p_l(z)t_{ij}(z)\frac{dz}{z}=
    \sum\limits_{l=0}^{\infty}u^{-l}\frac{1}{2\pi}\int\limits_{0}^{2\pi} p_l(e^{i\phi})t_{ij}(e^{i\phi})d\phi\\
    &=\sum\limits_{l=0}^{\infty}u^{-l}\sum\limits_{m=0}^{\infty} t_{ij}^{(m)}\sum\limits_{k=0}^l p_k^l\frac{1}{2\pi}\int\limits_{0}^{2\pi}e^{i(k-m)\phi}d\phi=\sum\limits_{l,m=0}^{\infty}u^{-l}t_{ij}^{(m)}\sum\limits_{k=0}^l p_k^l \delta_{km}\\
    &=\sum\limits_{l=0}^{\infty}u^{-l}\sum\limits_{m=0}^{l}p_m^l t_{ij}^{(m)},
\end{eqnarray}
where $p_l^{l}=1$ (as the highest coefficient of a monic orthogonal polynomial).

Let us notice that $W$ can be described by a $1$-sided infinite upper triangular matrix with $1$ on the diagonal and with finitely many nonzero entries in each row. Hence, we conclude that $W$ is invertible. 
It remains to show that $W$ is an algebraic homomorphism.  Consider the generating function $\tilde{T}$ of the algebra $OY(\gl_n,\{p_m\})$. We need to show that $W(\tilde{T})$  satisfies the equation \eqref{eqn:GenYangian_2}.

We have 
\begin{eqnarray}
 &\left(u+L^a_u+\frac{1}{u}L^{s(b)}_u - v-L^a_v-\frac{1}{v}L^{s(b)}_v\right) [W(\tilde{T})_{ij}(u),W(\tilde{T})_{kl}(v)]\nonumber\\
 &=[\left(u+L^a_u+\frac{1}{u}L^{s(b)}_u\right)W(\tilde{T})_{ij}(u),W(\tilde{T})_{kl}(v)]\nonumber\\
 &-[W(\tilde{T})_{ij}(u),(v+L^a_v+\frac{1}{v}L^{s(b)}_v)W(\tilde{T})_{kl}(v)]\nonumber\\
 &=\frac{1}{(2\pi i)^2} \int\limits_{|z|=1}\int\limits_{|w|=1}K(z,\frac{1}{u})K(w,\frac{1}{v})[T_{ij}(z),T_{kl}(w)](\frac{1}{w}-\frac{1}{z})\,dz\,dw\nonumber\\
 &=\frac{1}{(2\pi i)^2} \int\limits_{|z|=1}\int\limits_{|w|=1}K(z,\frac{1}{u})K(w,\frac{1}{v})\frac{T_{kj}(z)T_{il}(w)-T_{kj}(w)T_{il}(z)}{z-w}(\frac{1}{w}-\frac{1}{z})\,dz\,dw\nonumber\\
 &=W(\tilde{T})_{kj}(u)W(\tilde{T})_{il}(v)-W(\tilde{T})_{kj}(v)W(\tilde{T})_{il}(u)
\end{eqnarray}
where the second equality follows from identity \eqref{eqn:GenFunctionPol_1} and third one from definition of $Y(\gl_n)$.
\end{proof}
Let us denote 
\[
\omega(u):=tr \tilde{T}(u)=\sum\limits_{i=1}^n \tilde{T}_{ii}(u).
\]
Then it  follows immediately from Theorem~\ref{thm:ImageYangian} and the corresponding fact for $Y(\gl_n)$ that
\begin{cor}\label{cor:FirstIntegrals}
    \[
    [\omega(u),\omega(v)]=0.
    \]
\end{cor}
Hence $\omega(u)$ can be viewed as a generating function of integrals of motion of some quantum system.
\begin{prop}\label{prop:W_Inverse}
The inverse operator $W^{-1}:Y(\gl_n)\to OY(\gl_n,\{p_m\})$ is  given by the formula
\[
t_{ij}^{(r)}\to \tilde{t}_{ij}^{(r)}+\sum\limits_{l=0}^{r-1} q_l^r \tilde{t}_{ij}^{(l)},r\geq 1.
\]
where coefficients $\{q_l^r\}_{l=0}^r$ (where we put $q_r^r=1$) are defined by the system
\begin{equation}
 x^r=\sum\limits_{l=0}^r q_l^r p_l(x), r\geq 0.\label{eqn:W_Coeff}   
\end{equation}
Furthermore, it can be written as an integral operator
\begin{equation}\label{eqn:W_Inverse_IntegralForm}
W^{-1}(t)_{ij}(u):=\frac{1}{2\pi i}\int\limits_{|z|=1} K^{-1}(z,\frac{1}{u})\tilde{t}_{ij}(z)\frac{dz}{z},i,j=1\ldots,n. 
\end{equation} 
with kernel
\[
K^{-1}(z,u)=\sum\limits_{l=0}^{\infty} u^{l}\sum\limits_{m=0}^{l}q_{m}^{l}z^{m}
\]
\end{prop}
\begin{proof}
    The invertibility of $W$ and formula \eqref{eqn:W_Coeff} follow from the definition of $W$. Formula \eqref{eqn:W_Inverse_IntegralForm} can be shown similarly to the formula \eqref{eqn:W_IntegralForm}.
\end{proof}

\begin{rem}
The coefficients $\{q_l^r\}_{l=0}^r, r\geq 0$ satisfy recurrence relations
\begin{equation}\label{eqn:InvRecurence_1}
q_m^{r+1}=q_{m-1}^r+a_m q_m^r+b_{m+1} q_{m+1}^r,\, m=0,\ldots, r+1,\, r\geq 0.
\end{equation}
Here we use the convention  $q_m^r:=0$ if $m<0$ or $m>r$. Denote by 
\[
q_l(x):=\sum\limits_{m=0}^{l}q_{m}^{l}z^{m}, l\geq 0
\]
the corresponding monic polynomials. For any sequence $\{c_n\}_{n=0}^{\infty}$ of complex numbers define
the following (pseudo) differential operator:
\[
\tilde{L}^c_z:\suml_{k=0}^{\infty} d_k z^{k}\mapsto \suml_{k=0}^{\infty} c_k d_k z^{k}.
\]
 Then the system \eqref{eqn:InvRecurence_1} can be rewritten as follows
\begin{equation}\label{eqn:InvRecurence_2}
q_{r+1}=x q_r+ \tilde{L}_x^{a}q_r+\frac{1}{x}\tilde{L}_x^{b}q_r,\,r\geq 0.
\end{equation}
Note that the sequence $\{b_n\}_{n=0}^{\infty}$ starts from $b_0=0$ (see \eqref{eqn:abInitCondition}), and consequently, last term is a polynomial. Hence, the polynomials $\{q_l\}_{l=0}^{\infty}$ are not  orthogonal in general.
\end{rem}

\begin{defin}
Let $R\in (End(\mathbb{C}^n))^{\otimes 2}$ be the   $R$- matrix:
\[
R(u):=I-\frac{1}{u}P,
\]
 where $P$ is the permutation operator. 
 We introduce 
 $R(u_1,u_2\ldots,u_m)\in (End(\mathbb{C}^n))^{\otimes m}$ as follows:
\begin{equation}
    R(u_1,u_2\ldots,u_m):=R_{m-1,m}\left(R_{m-2,m}R_{m-2,m-1}\right)\ldots \left(R_{1,m}\ldots R_{12}\right),
\end{equation}  
where $R_{ij}=R_{ij}(u_i-u_j)$ and $i,j$ are indexes of two copies of $End(\mathbb{C}^n)$ on which operator $R_{ij}$ acts.
\end{defin}

We immediately have
\begin{cor}
\begin{eqnarray}   &\int\limits_{(\mathbb{S}^1)^m}\prod\limits_{i=1}^m \frac{K^{-1}(z_i,u_i)}{z_i} R(u_1,u_2\ldots,u_m)\tilde{T}_1(z_1)\ldots\tilde{T}_m(z_m)\,dz_1\ldots \,dz_m\nonumber\\
&=\int\limits_{(\mathbb{S}^1)^m}\tilde{T}_m(z_m)\ldots\tilde{T}_1(z_1)\prod\limits_{i=1}^m \frac{K^{-1}(z_i,u_i)}{z_i}R(u_1,u_2,\ldots,u_m)\,dz_1\ldots \,dz_m
\end{eqnarray}  
\end{cor}

Also, Theorem~\ref{thm:ImageYangian} immediately implies
\begin{cor}
 $OY(\gl_n,\{p_m\})$ is a Hopf algebra with product
 \[
 u\cdot_W v:= W^{-1}(W(u)\cdot W(v)), u,v\in OY(\gl_n,\{p_m\}),
 \]
 where $\cdot$ inside the brackets is the product in $Y(\gl_n)$), the coproduct $\Delta_{OY(\gl_n,\{p_m\})}$ and the antipode $S_{OY(\gl_n,\{p_m\})}$ are defined by
 \[
 \Delta_{OY(\gl_n,\{p_m\})}:=(W^{-1}\otimes W^{-1})\circ \Delta\circ W, \,\, S_{OY(\gl_n,\{p_m\})}:= W^{-1}\circ S\circ W,
 \]
 and $\Delta$ and $S$ are the coproduct and the antipode for $Y(\gl_n)$ respectively.
\end{cor}
\end{paragraph}


\section{Quantum Christoffel-Darboux formula}

The Christoffel-Darboux formula  for monic orthogonal polynomials $\{p_l\}_{l=1}^{\infty}$  satisfying the recurrence relation \eqref{eqn:Recurrence} can be formulated as follows
\begin{equation}
    (x-y)\sum\limits_{k=1}^n\frac{p_k(x)p_k(y)}{b_1 b_2\cdots b_k}=\frac{1}{b_1 b_2\cdots b_n}\left(p_{n+1}(x)p_n(y)-p_{n}(x)p_{n+1}(y)\right)
\end{equation}
Since (informally speaking) orthogonal polynomials $\{p_l\}_{l=1}^{\infty}$ correspond to generators $t_{\cdot\cdot}^{l}$ of Yangian $OY(\gl_n,a,b)$, it seems natural to deduce an analogue of Christoffel-Darboux (C.-D.) formula for the $OY(\gl_n,a,b)$. The C.-D. formula uses orthogonal polynomials of two independent variables $x$ and $y$. Hence we will work in the tensor product of two copies of Yangian.

Let $\tilde{t}_{ij}^{(r)},i,j=1,\ldots, N, r\in \mathbb{N}\cup\{0\}$ and  $\hat{t}_{ij}^{(r)},i,j=1,\ldots, N, r\in \mathbb{N}\cup\{0\}$ be generators of two copies of Yangian $OY(\gl_N,a,b)$. Define 
\[
(Z^n)_{ijkl}:=\sum\limits_{m=1}^n\frac{1}{b_1b_2\cdots b_m}\tilde{t}_{ij}^{(m)}\otimes \hat{t}_{kl}^{(m)}\in OY(\gl_N,a,b)\otimes OY(\gl_N,a,b),n\in\mathbb{N},i,j,k,l=0,\ldots,N,
\]
which serves as an analogue of $\sum\limits_{k=1}^n\frac{p_k(x)p_k(y)}{b_1 b_2\cdots b_k}$ and
\[
A_{\alpha\beta\epsilon \kappa}^{s}:=\tilde{t}_{\alpha\beta}^{(s)}\otimes 1-1\otimes \hat{t}_{\epsilon \kappa}^{(s)},s\in \mathbb{N}\cup\{0\}, A_{\alpha\beta\epsilon \kappa}^{-1}:=0,\,\alpha,\beta,\epsilon,\kappa=0,\ldots,N,
\]
 which is an analogue of $p_s(x)-p_s(y)$ (it is equal to $(x-y)$ if $s=1$). Then 
\begin{lem}
    We have
    
    \begin{eqnarray}
        [A_{\alpha\beta\epsilon \kappa}^{s+1} &+ a_s A_{\alpha\beta\epsilon \kappa}^{s} + b_s A_{\alpha\beta\epsilon \kappa}^{s-1},Z^n_{ijkl} ] =[A_{\alpha\beta\epsilon \kappa}^{s},\sum\limits_{m=1}^n \frac{a_m}{b_1b_2\cdots b_m}\tilde{t}_{ij}^{(m)}\otimes \tilde{t}_{kl}^{(m)}]\nonumber\\
&+[A_{\alpha\beta\epsilon \kappa}^{s},\sum\limits_{m=1}^{n-1} \frac{1}{b_1b_2\cdots b_m}(\tilde{t}_{ij}^{(m+1)}\otimes \tilde{t}_{kl}^{(m)}+\tilde{t}_{ij}^{(m)}\otimes \tilde{t}_{kl}^{(m+1)})]
        \nonumber\\
        &+[\tilde{t}_{\alpha\beta}^{(s)}\otimes 1,\frac{1}{b_1b_2\cdots b_n}\tilde{t}_{ij}^{(n+1)}\otimes \hat{t}_{kl}^{(n)}]\nonumber\\
        &-[1\otimes \hat{t}_{\epsilon \kappa}^{(s)},\frac{1}{b_1b_2\cdots b_n}\tilde{t}_{ij}^{(n)}\otimes \hat{t}_{kl}^{(n+1)}]\nonumber\\
        &+(\tilde{t}_{i\beta}^{(s)}\otimes 1)Z^n_{\alpha jkl}- Z^n_{i\beta kl}(\tilde{t}_{\alpha j}^{(s)}\otimes 1)\nonumber\\
        &+Z^n_{ijk\kappa}(1\otimes \hat{t}_{\epsilon l}^{(s)})
        -(1\otimes \hat{t}_{k \kappa}^{(s)})Z^n_{ij\epsilon l}.\label{eqn:CristoffelDarbouxFormula}
    \end{eqnarray}
\end{lem}
\begin{proof}
    Immediately follows from commutation relations defining Yangian.
\end{proof}
\begin{cor}
    \begin{equation}
      [A_{\alpha\beta\epsilon \kappa}^{1},Z^n_{ijkl} ] =  \delta_{i\beta}Z^n_{\alpha jkl}- Z^n_{i\beta kl}\delta_{\alpha j}+Z^n_{ijk\kappa}\delta_{\epsilon l}
        -\delta_{k \kappa}Z^n_{ij\epsilon l}.
    \end{equation}
\end{cor}
\begin{proof}
    It is enough to put $s=0$ in \eqref{eqn:CristoffelDarbouxFormula}.
\end{proof}
\begin{cor}\label{cor:CrDarFormula_2}
    If $a_m=a,m\in \mathbb{N}$ we have 
 \begin{eqnarray}
        [A_{\alpha\beta\epsilon \kappa}^{s+1} + b_s A_{\alpha\beta\epsilon \kappa}^{s-1},&Z^n_{ijkl} ] 
        = \frac{1}{b_1b_2\cdots b_n}\left( [\tilde{t}_{\alpha\beta}^{(s)}\otimes 1,\tilde{t}_{ij}^{(n+1)}\otimes\hat{t}_{kl}^{(n)}]-[1\otimes \hat{t}_{\epsilon \kappa}^{(s)},\tilde{t}_{ij}^{(n)}\otimes \hat{t}_{kl}^{(n+1)}]\right)\nonumber\\
        &+\left[A_{\alpha\beta\epsilon \kappa}^{s},\sum\limits_{m=1}^{n-1} \frac{1}{b_1b_2\cdots b_m} \left(\tilde{t}_{ij}^{(m+1)}\otimes \tilde{t}_{kl}^{(m)}+\tilde{t}_{ij}^{(m)}\otimes \tilde{t}_{kl}^{(m+1)}\right)\right]\nonumber\\
        &+(\tilde{t}_{i\beta}^{(s)}\otimes 1)Z^n_{\alpha jkl}- Z^n_{i\beta kl}(\tilde{t}_{\alpha j}^{(s)}\otimes 1)+Z^n_{ijk\kappa}(1\otimes \hat{t}_{\epsilon l}^{(s)})
        -(1\otimes \hat{t}_{k \kappa}^{(s)})Z^n_{ij\epsilon l}.
    \end{eqnarray}
In particular, if $(\alpha,\beta,\epsilon, \kappa)=(i,j,k,l)$ we have
 \begin{eqnarray}
        [A_{ijkl}^{s+1} -A_{ijkl}^{s}+ b_s A_{ijkl}^{s-1}, \, Z^n_{ijkl} ] &=\frac{1}{b_1b_2\cdots b_n}\left( [\tilde{t}_{ij}^{(s)}\otimes 1,\tilde{t}_{ij}^{(n+1)}\otimes\hat{t}_{kl}^{(n)}]-[1\otimes \hat{t}_{kl}^{(s)},\tilde{t}_{ij}^{(n)}\otimes \hat{t}_{kl}^{(n+1)}]\right)\nonumber\\
        &+\left[A_{ijkl}^{s},\sum\limits_{m=1}^{n-1} \frac{1}{b_1b_2\cdots b_m} \left(\tilde{t}_{ij}^{(m+1)}\otimes \tilde{t}_{kl}^{(m)}+\tilde{t}_{ij}^{(m)}\otimes \tilde{t}_{kl}^{(m+1)}\right)\right].
    \end{eqnarray}
\end{cor}
In the case of Chebyshev polynomials i.e. $a_m=a,b_m=b, b\neq 0$ previous corollary can be improved. Define
\[
Z^{n,p,q}_{ijkl}:=\sum\limits_{m=1}^n\frac{1}{b^m}\tilde{t}_{ij}^{(m+p)}\otimes \hat{t}_{kl}^{(m+q)}\in OY(\gl_N,a,b)\otimes OY(\gl_N,a,b),n\in\mathbb{N},p,q\in\mathbb{Z},i,j,k,l=0,\ldots,N.
\]
Then we have:
\begin{cor}\label{cor:CrDarFormula_Cheb}
\begin{eqnarray}
        [A_{\alpha\beta\epsilon \kappa}^{s+1} &+ b A_{\alpha\beta\epsilon \kappa}^{s-1},Z^{n,p,q}_{ijkl} ] 
        = \frac{1}{b^n}\left( [\tilde{t}_{\alpha\beta}^{(s)}\otimes 1,\tilde{t}_{ij}^{(n+p+1)}\otimes\hat{t}_{kl}^{(n+q)}]-[1\otimes \hat{t}_{\epsilon \kappa}^{(s)},\tilde{t}_{ij}^{(n+p)}\otimes \hat{t}_{kl}^{(n+q+1)}]\right)\nonumber\\
        &+\left( [\tilde{t}_{\alpha\beta}^{(s)}\otimes 1,\tilde{t}_{ij}^{(p)}\otimes\hat{t}_{kl}^{(q+1)}]-[1\otimes \hat{t}_{\epsilon \kappa}^{(s)},\tilde{t}_{ij}^{(p+1)}\otimes \hat{t}_{kl}^{(q)}]\right)+\left[A_{\alpha\beta\epsilon \kappa}^{s},Z^{n-1,p+1,q}_{ijkl}+Z^{n-1,p,q+1}_{ijkl}\right]\nonumber\\
        &+(\tilde{t}_{i\beta}^{(s)}\otimes 1)Z^{n,p,q}_{\alpha jkl}- Z^{n,p,q}_{i\beta kl}(\tilde{t}_{\alpha j}^{(s)}\otimes 1)+Z^{n,p,q}_{ijk\kappa}(1\otimes \hat{t}_{\epsilon l}^{(s)})
        -(1\otimes \hat{t}_{k \kappa}^{(s)})Z^{n,p,q}_{ij\epsilon l}.
    \end{eqnarray}
In particular, if $(\alpha,\beta,\epsilon, \kappa)=(i,j,k,l)$ we have
\begin{eqnarray}
        [A_{ijkl}^{s+1} &-A_{ijkl}^{s}+ b A_{ijkl}^{s-1},Z^{n,p,q}_{ijkl} ] 
        = \frac{1}{b^n}\left( [\tilde{t}_{ij}^{(s)}\otimes 1,\tilde{t}_{ij}^{(n+p+1)}\otimes\hat{t}_{kl}^{(n+q)}]-[1\otimes \hat{t}_{kl}^{(s)},\tilde{t}_{ij}^{(n+p)}\otimes \hat{t}_{kl}^{(n+q+1)}]\right)\nonumber\\
        &+\left( [\tilde{t}_{ij}^{(s)}\otimes 1,\tilde{t}_{ij}^{(p)}\otimes\hat{t}_{kl}^{(q+1)}]-[1\otimes \hat{t}_{kl}^{(s)},\tilde{t}_{ij}^{(p+1)}\otimes \hat{t}_{kl}^{(q)}]\right)+\left[A_{ijkl}^{s},Z^{n-1,p+1,q}_{ijkl}+Z^{n-1,p,q+1}_{ijkl}\right].\nonumber
    \end{eqnarray}
\end{cor}

\section{Case of Dickson polynomials}\label{sec:Chebyshev}
Let us look at the case of Dickson polynomials, that is $a_n=\alpha,n\in\mathbb{N}\cup\{0\}$, $b_n=\beta,n\in\mathbb{N}$, $\beta\neq 0$. Dickson polynomials $\{d_n^k\}_{n\geq 0}$ of ($k+1$) type are defined by the system of recurrent relations (\cite{WY2012,Stoll2008})
\[
\begin{array}{cc}
    d_0^k &= 2-k,\,\,\,\, d_k^1=x,\,\,\,\ k\in\mathbb{N}\cup {0} \\
    d_{n+1}^k &= x d_n -\beta d_{n-1}^k
\end{array}
\]
\begin{rem}
 In the case of $k=0$ or $k=1$ with $\beta=1$ they correspond to the Chebyshev polynomials of first (corr., second) type.   
\end{rem}
 We can assume without loss of generality that $k=1$ i.e. we consider Dickson polynomials of  the second type. Indeed, since relations  \eqref{eqn:GenYangian_1} are invariant with respect to multiplication on the nonzero constant, Dickson polynomials of the other types will lead to the same algebra unless $k=2$. The case $k=2$ will be considered elsewhere.

We will use the notation $OY_{\beta}(\gl_n)$ for the corresponding Yangian.
Then $L_{\cdot}^{a}=\alpha Id$, $L^{s(b)}=\beta Id$ and system \eqref{eqn:GenYangian_2} becomes
\[
\left(u+\frac{\beta}{u}-v-\frac{\beta}{v}\right)[\tilde{T}_{ij}(u),\tilde{T}_{kl}(v)]=\tilde{T}_{kj}(u)\tilde{T}_{il}(v)-\tilde{T}_{kj}(v)\tilde{T}_{il}(u).i,j,k,l=1,\ldots,n
\]
which is nothing else but the inhomogeneous RTT relation 
\begin{equation}\label{eqn:RTTChebyshev}
R^{\beta}(u,v)\tilde{T}_1(u)\tilde{T}_2(v)=\tilde{T}_2(v)\tilde{T}_1(u)R^{\beta}(u,v)
\end{equation}
with inhomogeneous $R^{\beta}$ matrix:
\[
R^{\beta}(u,v)=I-\frac{1}{u+\frac{\beta}{u}-v-\frac{\beta}{v}}P.
\]
The corresponding map $W:OY_{\beta}(\gl_n)\mapsto Y(\gl_n) $  will have the form:
\begin{equation}\label{eqn:W_Cheb}
W(\tilde{t})_{ij}(u):=\frac{1}{2\pi i}\int\limits_{|z|=1} \frac{ u^2}{ u^2-u (z-\alpha) +\beta}t_{ij}(z)\frac{dz}{z}=\frac{u}{u+\frac{\beta}{u}+\alpha}t_{ij}(u+\frac{\beta}{u}+\alpha),i,j=1\ldots,n.
\end{equation}
\begin{rem}
Notice that from the RTT relation follows that $R^{\beta}$ is a solution of general Yang-Baxter equation:
\[
R_{12}(u,v)R_{13}(u,w)R_{23}(v,w)=R_{23}(v,w)R_{13}(u,w)R_{12}(u,v).
\]  
\end{rem}
In this case we have the following proposition
\begin{prop}\label{prop:CommutationRelations}
The system of relations \eqref{eqn:RTTChebyshev} is equivalent to the system
\[
[\tilde{t}_{ij}^{(r)},\tilde{t}^{(s)}_{kl}]=\sum\limits_{p,m=0}^{m-s\leq p\leq r-m-1}\beta^m\left(\tilde{t}^{(r-p-m-1)}_{kj}\tilde{t}^{(p+s-m)}_{il}-\tilde{t}^{(p+s-m)}_{kj}\tilde{t}^{(r-p-m-1)}_{il}\right).
\]
\end{prop}
\begin{proof}
Multiplication of both sides of the system \eqref{eqn:RTTChebyshev} on formal series
$\sum\limits_{p,m=0}^{\infty}\beta^m u^{-p-m-1}v^{p-m}$ and equating coefficients gives the result.
\end{proof}
Similarly to the Yangian we can define the evaluation homomorphism:
\begin{prop}\label{prop:SurjectiveHomomorphism}
\begin{itemize}
    \item[(i)] 
    The map defined by
    \begin{equation}\label{eqn:DefSurHom_1}
    \tilde{T}_{ij} (u)\mapsto \delta_{ij}+ E_{ij}\frac{1}{u+\frac{\beta}{u}}   
    \end{equation}
    is a surjective homomorphism $OY_{\beta}(\gl_n)\to U(\gl_n)$. The map defined by
    \begin{equation}\label{eqn:DefSurHom_2}
    E_{ij}\mapsto \tilde{t}_{ij}^{(1)}
    \end{equation}
    is an embedding $U(\gl_n)\to OY_{\beta}(\gl_n)$.
\item[(ii)]   The map defined by
\begin{equation}\label{eqn:Homomorf_1}
    \tilde{T}_{ij} (u)\mapsto T_{ij}(u+\frac{\beta}{u})      \end{equation}
    is a  homomorphism $OY_{\beta}(\gl_n)[u^{-1}]\to Y(\gl_n)[u^{-1}]$.
    \item[(iii)] Let
 \[
 f(u)=1+\sum\limits_{k=1}^{\infty}\frac{f_k}{u^k}\in \mathbb{C}[[u^{-1}]],
 \]
 and $B$ any invertible $n\times n$ matrix. Then we have the following automorphisms of $OY_{\beta}(\gl_n)$:
 \begin{itemize}
 \item[(a)]
 \begin{equation}
\tilde{T}(u)\mapsto f(u)\tilde{T}(u).
\end{equation}

\item[(b)]
\begin{equation}
    \tilde{T}(u)\mapsto B \tilde{T}(u) B^{-1}.
\end{equation}
\end{itemize}
    \end{itemize}
\end{prop}
\begin{proof}
By the definition of $OY_{\beta}(\gl_n)$ we need to check that
    \[
\begin{aligned}
    \left(u+\frac{\beta}{u}-v-\frac{\beta}{v}\right)&[E_{ij},E_{kl}](u+\frac{\beta}{u})^{-1}(v+\frac{\beta}{v})^{-1}=\\
    &(\delta_{kj}+ \frac{E_{kj}}{u+\frac{\beta}{u}})(\delta_{il}+ \frac{E_{il}}{v+\frac{\beta}{v}})-(\delta_{kj}+ \frac{E_{kj}}{v+\frac{\beta}{v}})(\delta_{il}+ \frac{E_{il}}{u+\frac{\beta}{u}}).
\end{aligned}
\]
This follows from commutation relations in $\gl_n$. Now we put $r=s=1$  in Proposition \ref{prop:CommutationRelations} to get
\[
[\tilde{t}_{ij}^{(1)},\tilde{t}^{(1)}_{kl}]=\delta_{kj}\tilde{t}_{il}^{(1)}-\delta_{il} \tilde{t}_{kj}^{(1)},
\]
and the statement (i) follows. Statement (ii) follows from formula \eqref{eqn:W_Cheb}, while (iii) follows from  \cite[Proposition 1.3.1]{Molevbook}.

\end{proof}
\begin{rem}
    Notice that formally speaking the RTT relation \eqref{eqn:RTTChebyshev}  is invariant with respect to transformation
\[
T(u)\mapsto T(\frac{\beta}{u})
\]
Of course, we cannot define automorphism of $OY_{\beta}(\gl_n)$ because $T(\frac{\beta}{u})$ is a series in positive powers of $u$. This leads to necessity to define  an algebra $\widehat{OY}_{\beta}(\gl_n)$ with relation \eqref{eqn:RTTChebyshev} where $T\in \widehat{OY}_{\beta}(\gl_n)[u,u^{-1}]$ is a Laurent series in $u$. The algebra is not trivial since we have homomorphism
\[
H:\widehat{OY}_{\beta}(\gl_n)\to OY_{\beta}(\gl_n), \,\, \hat{t}_{ij}^{(k)}\to 
\left\{
\begin{array}{cc}
   t_{ij}^{(k)}  & k\geq 0 \\
    0 & k<0
\end{array}
\right.
\]
Further investigation of this algebra will be done elsewhere.
\end{rem}

\begin{prop}
Let  $c\in\mathbb{C}$. We have the following embedding of $OY_{\beta}(\gl_n)$ into its completion:

\begin{equation}
\tilde{T}(u)\mapsto \tilde{T}(\phi_c(u)),
\end{equation}
where $\phi_c\in \mathbb{C}[[u^{-1}]]$ is a holomorphic family of functions defined by the equation
\begin{equation}\label{eqn:Phi_c_def}
\phi_c(u)+\frac{\beta}{\phi_c(u)}=u+\frac{\beta}{u}+c, 
\end{equation}
which satisfies the group property
\[
\phi_{c+d}=\phi_c\circ\phi_d, c,d\in \mathbb{C}.
\]
\end{prop}
\begin{proof}

Let us show that the equation \eqref{eqn:Phi_c_def} has a unique solution in $\mathbb{C}[[u^{-1}]]$. We have that 
\begin{equation}
\phi_c(u)=\frac{u}{2}\left(1+\frac{c}{u}+\frac{\beta}{u^2}\pm\sqrt{(1+\frac{c}{u}+\frac{\beta}{u^2})^2}-\frac{4\beta}{u^2}\right).\label{eqn:Phi_Variants}
\end{equation}
We search for the series $\{a_k\}_{k=2}^{\infty}$ such that
\[
\left(1+\frac{c}{u}+\sum\limits_{k=2}^{\infty}\frac{a_k}{u^k}\right)^2=\left(1+\frac{c}{u}+\frac{\beta}{u^2}\right)^2-\frac{4\beta}{u^2}.
\]
This equation is equivalent to the system
\begin{equation}
    \left\{
\begin{array}{cc}
     &a_2 = -\beta ,a_3 =2\beta c, a_4 = -2\beta c^2\\
     &a_n = -c a_{n-1}-\frac{1}{2}\sum\limits_{m,l\geq 2, m+l=n} a_m a_l,n\geq 5,
\end{array}
    \right.
\end{equation}
which has a unique solution. Thus, choosing minus sign in \eqref{eqn:Phi_Variants} we get that
\begin{equation}
   \phi_c(u)= \frac{u}{2}\left(\frac{\beta}{u^2}-\sum\limits_{k=2}^{\infty}\frac{a_k}{u^k}\right)
\end{equation}
is a unique solution of the equation \eqref{eqn:Phi_c_def} in $\mathbb{C}[[u^{-1}]]$. The group property follows from 
 the uniqueness.

\end{proof}
\subsection{Quantum determinant}

\begin{defin} Let $R^{\beta}\in (End(\mathbb{C}^n))^{\otimes m}$ a function defined by
\begin{equation}
    R^{\beta}(u_1,u_2\ldots,u_m):=R^{\beta}_{m-1,m}\left(R^{\beta}_{m-2,m}R^{\beta}_{m-2,m-1}\right)\ldots \left(R^{\beta}_{1,m}\ldots R^{\beta}_{12}\right)
\end{equation}  
where we abbreviate $R^{\beta}_{ij}=R^{\beta}_{ij}(u_i,u_j)$.
\end{defin}
Then we have
\begin{prop}\label{prop:Qdet_principal}
We have the identity
\begin{equation}
    R^{\beta}(u_1,u_2\ldots,u_m)\tilde{T}_1(u_1)\ldots \tilde{T}_{m}(u_m)=\tilde{T}_{m}(u_m)\ldots \tilde{T}_1(u_1) R^{\beta}(u_1,u_2\ldots,u_m).
\end{equation}
\end{prop}
\begin{proof}
    Proposition 1.6.1 in \cite{Molevbook}.
\end{proof}
Define
\[
A_m:(\mathbb{C}^n)^{\otimes m}\mapsto (\mathbb{C}^n)^{\otimes m}, A_m(e_{i_1}\otimes\ldots \otimes e_{i_m}):=\sum\limits_{p\in S_m} \sgn(p) e_{i_{p(1)}}\otimes\ldots \otimes e_{i_{p(m)}}.
\]
\begin{prop}
If $u_{i+1}=\phi_{-1}(u_i)$, for all $i=1,\ldots, m-1$ then
 $R^{\beta}(u_1,u_2\ldots,u_m)=A_m$.
\end{prop}
\begin{proof}
    The result follows from Proposition 1.6.2 in \cite{Molevbook}.
\end{proof}
\begin{cor}
    We have 
    \begin{equation}\label{eqn:Qdet_1}
    A_m\tilde{T}_1(u)\tilde{T}_2(\phi_{-1}(u))\ldots  \tilde{T}_{m}(\phi_{-(m-1)}(u))=\tilde{T}_{m}(\phi_{-(m-1)}(u))\ldots \tilde{T}_2(\phi_{-1}(u))\tilde{T}_1(u) A_m.
    \end{equation}
\end{cor}
\begin{defin}
    The quantum determinant of the matrix $\tilde{T}(u)$ with the coefficients in $OY_{\beta}(\gl_n)$ is the formal series
    \[
    qdet\,\tilde{T}(u)=1+ d_1 u^{-1}+ d_2 u^{-2}+\ldots
    \]
    such that the element \eqref{eqn:Qdet_1} with $m=n$ is equal $A_n qdet\,\tilde{T}(u)$.
\end{defin}
\begin{prop}
    For any permutation $\tau\in \mathcal{S}_{n}$ we have
    \begin{eqnarray}
     qdet\,\tilde{T}(u) &= sgn(\tau)\sum\limits_{\sigma\in \mathcal{S}_{n}} sgn(\sigma) \tilde{t}_{\sigma(1),\tau(1)}(u) \tilde{t}_{\sigma(2),\tau(2)}(\phi_{-1}(u))\ldots \tilde{t}_{\sigma(n),\tau(n)}(\phi_{-(n-1)}(u))  \nonumber\\
     &= sgn(\tau)\sum\limits_{\sigma\in \mathcal{S}_{n}} sgn(\sigma) \tilde{t}_{\tau(1),\sigma(1)}(u) \tilde{t}_{\tau(2),\sigma(2)}(\phi_{-1}(u))\ldots \tilde{t}_{\tau(n),\sigma(n)}(\phi_{-(n-1)}(u)).\nonumber
    \end{eqnarray}
        
\end{prop}
\begin{defin}
    The quantum $m\times m$ minors $\tilde{t}_{b_1\ldots b_m}^{a_1\ldots a_m}=\tilde{t}_{b_1\ldots b_m}^{a_1\ldots a_m}(u)$ of matrix $\tilde{T}(u)$ we define as matrix elements of operator 
    \[
    A_m\tilde{T}_1(u)\tilde{T}_2(\phi_{-1}(u))\ldots  \tilde{T}_{m}(\phi_{-(m-1)}(u))
    \] i.e.
\[
A_m\tilde{T}_1(u)\tilde{T}_2(\phi_{-1}(u))\ldots  \tilde{T}_{m}(\phi_{-(m-1)}(u))=\sum\limits_{a_i,b_j=1}^{n}e_{a_1b_1}\otimes \ldots \otimes e_{a_m b_m}\otimes \tilde{t}_{b_1\ldots b_m}^{a_1\ldots a_m}(u).
\]    
\end{defin}
\begin{prop}\label{prop:comult_min}
Comultiplication of quantum minors is given by the formula:
\begin{equation}
\Delta(\tilde{t}_{b_1\ldots b_m}^{a_1\ldots a_m}(u))=\sum\limits_{c_1<\ldots< c_m} \tilde{t}_{c_1\ldots c_m}^{a_1\ldots a_m}(u)\otimes \tilde{t}_{b_1\ldots b_m}^{c_1\ldots c_m}(u).
\end{equation}
\end{prop}
\begin{cor}
    The quantum determinant is comultiplicative i.e.
    \[
    \Delta(qdet\,\tilde{T}(u))=qdet\,\tilde{T}(u)\otimes qdet\,\tilde{T}(u).
    \]
\end{cor}
\begin{proof}
    It follows from the proposition \ref{prop:comult_min} and equality $qdet\,\tilde{T}(u)=\tilde{t}_{1\ldots n}^{1\ldots n}(u)$.
\end{proof}
\begin{prop}\label{prop:Qdet_2}
    \begin{equation}
        (u+\frac{\beta}{u}-v-\frac{\beta}{v})[\tilde{t}_{kl}(u),\tilde{t}_{b_1\ldots b_m}^{a_1\ldots a_m}(v)]=\sum\limits_{i=1}^m \tilde{t}_{a_i l}(u)\tilde{t}_{b_1\ldots b_m}^{a_1\ldots k\ldots a_m}(v)-\sum\limits_{i=1}^m \tilde{t}_{b_1\ldots l\ldots b_m}^{a_1\ldots a_m}(v) \tilde{t}_{k b_i}(u).
    \end{equation}
\end{prop}
\begin{proof}
Applying proposition \ref{prop:Qdet_principal} we get
\begin{align}
R^{\beta}(u,v,\phi_{-1}(v),\ldots,\phi_{-{m-1}}(v)) &\tilde{T}_0(u) \tilde{T}_1(v)\ldots \tilde{T}_m(\phi_{-{m-1}}(v))=\nonumber\\
&\tilde{T}_m(\phi_{-{m-1}}(v))\ldots \tilde{T}_1(v)\tilde{T}_0(u)R^{\beta}(u,v,\phi_{-1}(v),\ldots,\phi_{-{m-1}}(v))\label{eqn:Qdet_appl_1}.
\end{align}
On the other side
\[
R^{\beta}(u,v,\phi_{-1}(v),\ldots,\phi_{-{m-1}}(v))=A_m R^{\beta}_{0 m}(u,\phi_{-{m-1}}(v))\ldots R^{\beta}_{0 1}(u,v).
\]
Therefore,
\[
R^{\beta}(u,v,\phi_{-1}(v),\ldots,\phi_{-{m-1}}(v))=A_m \left(1-\frac{1}{u+\frac{\beta}{u}-v-\frac{\beta}{v}}(P_{0 1}+\ldots +P_{0 m})\right).
\]
Consequently, the result follows from the formula \eqref{eqn:Qdet_appl_1}.
\end{proof}
\begin{cor}\label{cor:QdetCenter_1}
    \[
    [\tilde{t}_{a_i b_j}(u),\tilde{t}_{b_1\ldots b_m}^{a_1\ldots a_m}(v)]=0, i,j=1,\ldots, m.
    \]
\end{cor}
\begin{proof}
    It follows from the anti-symmetry of $\tilde{t}_{b_1\ldots b_m}^{a_1\ldots a_m}$ ( w.r.t. the exchange of indices above and below) and Proposition \ref{prop:Qdet_2}.
\end{proof}
\begin{thm}\label{thm:QdetCenter}
    The center of $OY_{\beta}(gl_n)$ is generated by the coefficients of the series $qdet \tilde{T}(u)$.
\end{thm}
\begin{proof}
Corollary \eqref{cor:QdetCenter_1} implies that the coefficients of the series $qdet \tilde{T}(u)$ belong to the center.    
\end{proof}

\section{Further examples.}
\subsection{Case of Hermite Polynomials.}
The family of orthogonal Hermite polynomials corresponds to the choice $a_r=0$, $b_r=r$, $r\in\mathbb{N}$. In this case the equation for the generating function has the form: 
\[
\left(u+\frac{1}{u}-\partial_u-(v+\frac{1}{v}-\partial_v)\right)[\tilde{T}_{ij}(u),\tilde{T}_{kl}(v)]=\tilde{T}_{kj}(u)\tilde{T}_{il}(v)-\tilde{T}_{kj}(v)\tilde{T}_{il}(u).i,j,k,l=1,\ldots,n.
\]
Indeed, in this case we have $L^a=0$ and $L_z^{s(b)}=Id-z\partial_z$.

\begin{rem}
Equation for $\omega(u)= tr\,\tilde{T}(u) $ will be as follows:
\[
\left(u+\frac{1}{u}-\partial_u-(v+\frac{1}{v}-\partial_v)\right)[\omega(u), \omega(v)]=0.
\]
This identity allows us to show that 
\[
[\omega(u), \omega(v)]=0,
\]
without Corollary \ref{cor:FirstIntegrals}.
Indeed, denoting $B_{\phi,\psi}(u,v)=<[\omega(u), \omega(v)]\phi,\psi>$, we have 
\[
\left(u+\frac{1}{u}-\partial_u-(v+\frac{1}{v}-\partial_v)\right)B_{\phi,\psi}(u,v)=0
\]
and
\[
B_{\phi,\psi}(u,v)=-B_{\phi,\psi}(v,u).
\]
Hence, we have
\[
B_{\phi,\psi}(u,v)=C_{\phi,\psi}(u+v)uv e^{\frac{(u-v)^2}{4}}
\]
and, by the antisymmetry of $B$ we conclude that $B=0$. 
\end{rem}

\paragraph{Analogue of classical Yang-Baxter equation for the case of Hermite polynomials.}

Define 
\[
S(u):=(u+\frac{1}{u}-\partial_u)I-P, 
\]
\[
R(u,v):=\left[u+\frac{1}{u}-\partial_u-\left(v+\frac{1}{v}-\partial_v\right)\right])I-P.
\]
\begin{prop}
We have
\begin{equation}\label{eqn:YBeqn_1}
S_{12}(u)S_{13}(u+v)S_{23}(v)-S_{23}(v)S_{13}(u+v)S_{12}(u)=\frac{u^2+v^2+uv}{uv(u+v)}[S_{23}(v),S_{12}(u)],
\end{equation}
and
\begin{equation}\label{eqn:YBeqn_2}
R(u,v)=S(u-v)+\frac{uv-u^2-v^2}{uv(u-v)} I.
\end{equation}
\end{prop}
\begin{proof}
Equation \ref{eqn:YBeqn_1} immediately follows from classical Yang-Baxter equation, our definition of $S$ and change of variables formula.
\end{proof}

\begin{rem}
Another equivalent possibility is to work with functions
\[
\widehat{S}(u):=uS(u), \widehat{R}(u,v):=(u-v)R(u,v).
\]
 Then we have
\[
\widehat{S}_{12}(u)\widehat{S}_{13}(u+v)\widehat{S}_{23}(v)-\widehat{S}_{23}(v)\widehat{S}_{13}(u+v)\widehat{S}_{12}(u)=\left(1+\frac{u}{v}+\frac{v}{u}\right)[\widehat{S}_{23}(v),\widehat{S}_{12}(u)]
\]
 and
 \[
\widehat{R}(u,v)=\widehat{S}(u-v)+\left(1-\frac{u}{v}-\frac{v}{u}\right) I.
\]
\end{rem}
\subsection{Non-orthogonal polynomials.}
The framework of the section \ref{sec:OYangians} can be applied to the case of the families of nonorthogonal polynomials satisfying recurrence relation \eqref{eqn:Recurrence} as following example shows. 

Put 
\[
a_n=\left\{
\begin{array}{cc}
    a &  n-\mbox{even}\\
    -a & n-\mbox{odd}
\end{array}
\right., b_n=0,n\in\mathbb{N}\cup\{0\}
\]
in recurrence relation \eqref{eqn:Recurrence}. 
Then
\[
p_{2l}(x)=(x^2-a^2)^l, p_{2l+1}(x)=(x^2-a^2)^l (x+a), l\in \mathbb{N}\cup\{0\},
\]
with generating function
\[
K(z,u)=\frac{1+(z+a)u}{1-(z^2-a^2)u^2}
\]
and operator $L_z^{a}f=a f(-z)$. 

Therefore,
defining relations \eqref{eqn:GenYangian_2} of corresponding algebra will be
\begin{equation}\label{eqn:NonOrthog}
\left(u-v\right)[\tilde{T}_{ij}(u),\tilde{T}_{kl}(v)]+ a([\tilde{T}_{ij}(-u),\tilde{T}_{kl}(v)]-[\tilde{T}_{ij}(u),\tilde{T}_{kl}(-v)])=\tilde{T}_{kj}(u)\tilde{T}_{il}(v)-\tilde{T}_{kj}(v)\tilde{T}_{il}(u),
\end{equation}
where $i,j,k,l=1,\ldots,n$.

\subsection{$q$-Pochhammer symbols.}

In previous sections we used families of orthogonal polynomials in conjunction with the identity \eqref{eqn:GenYangian_0} to define a new family of realizations of the Yangian $Y(\gl_n)$. In this section we demonstrate how to deduce the defining relations for a family of algebras with the $q$-Pochhammer symbols instead of  monic orthogonal polynomials. It is an open problem to understand if the resulting algebras will be isomorphic to $Y(\gl_n)$.

Let
\[
\mathcal{P}_n^{q}(x)=(x;q)_n=\prod\limits_{k=0}^{n-1}(1-x q^k),\,\,\,  n\in\mathbb{N}\cup\{0\}
\]
be the $q$-Pochhammer symbol, considered as polynomial of degree $n$. Then we have
\[
x\mathcal{P}_n^{q}(x)=\frac{\mathcal{P}_n^{q}(x)-\mathcal{P}_{n+1}^{q}(x)}{q^n},\,\,\, n\in\mathbb{N}\cup\{0\}
\]
Applying the identity \eqref{eqn:GenYangian_0} with $f=\mathcal{P}_r^{q}$ and $g=\mathcal{P}_s^{q}$, $r,s\in \mathbb{N}$ and denoting
\[
m_{ij}^{(r)}:=(\mathcal{P}_r^{q}(E))_{ij},
\]
we get
\begin{equation}\label{eqn:Pochhammer_1}
(q^{-r}-q^{-s})[m^{(r)}_{ij},m^{(s)}_{kl}]+q^{-s}[m^{(r)}_{ij},m^{(s+1)}_{kl}]-q^{-r}[m^{(r+1)}_{ij},m^{(s)}_{kl}]=m^{(r)}_{kj}m^{(s)}_{il}-m^{(s)}_{kj}m^{(r)}_{il},
\end{equation}
where $i,j=1,\ldots,n,r,s\in \mathbb{N}$. We define the generating function as follows:
\[
m_{ij}(u)=\delta_{ij}+\sum\limits_{r=1}^{\infty}m^{(r)}_{ij} u^{-r}.
\]
Then we can rewrite \eqref{eqn:Pochhammer_1} in the form
\begin{equation}\label{eqn:Pochhammer_2}
(1-qu)[m_{ij}(qu),m_{kl}(v)]-(1-qv)[m_{ij}(u),m_{kl}(qv)]=m_{kj}(u)m_{il}(v)-m_{kj}(v)m_{il}(u),
\end{equation}
$i,j,k,l=1,\ldots,n$.
\begin{rem}
Notice that the limit $q\to 1$ corresponds to the Yangian $Y(\gl_n)$.
 Indeed, 
 for $q=1$ the mapping 
$t_{ij}(u)\mapsto m_{ij}(-u)$ defines an isomorphism from 
$Y(\gl_n) $ to the algebra generated by $m_{ij}^{(r)}$ subject to the relation 
\eqref{eqn:Pochhammer_1}.
\end{rem}
Define 
\[
\tilde{m}^n(u):=\frac{(1-u)m(u)}{\mathcal{P}_n^{q}(u)}
\]
Then it can be shown that
\begin{equation}\label{eqn:Pochhammer_3}
(1-q^n u)[\tilde{m}^n_{ij}(qu),\tilde{m}^n_{kl}(v)]-(1-q^n v)[\tilde{m}^n_{ij}(u),\tilde{m}^n_{kl}(qv)]=\tilde{m}^n_{kj}(u)\tilde{m}^n_{il}(v)-\tilde{m}^n_{kj}(v)\tilde{m}^n_{il}(u), 
\end{equation}
with $i,j,k,l=1,\ldots,n$. 
If $|q|<1$ we can take limit $n\to \infty$ and deduce that
\[
\tilde{m}(u):=\frac{(1-u)m(u)}{\mathcal{P}_{\infty}^{q}(u)}
\]
satisfies
\begin{equation}\label{eqn:Pochhammer_4}
[\tilde{m}_{ij}(qu),\tilde{m}_{kl}(v)]-[\tilde{m}_{ij}(u),\tilde{m}_{kl}(qv)]=\tilde{m}_{kj}(u)\tilde{m}_{il}(v)-\tilde{m}_{kj}(v)\tilde{m}_{il}(u), ,\, i,j,k,l=1,\ldots,n.
\end{equation}
\paragraph{Infinite Pochhammer symbol.}

Define
\[
m_{ij}^{(\infty)}:=(\mathcal{P}_{\infty}^{q}(E))_{ij}, \, i,j=1,\ldots,n.
\]
Assume $|q|>1$. Then taking $s\to\infty$ in \eqref{eqn:Pochhammer_1} we get
\begin{equation}\label{eqn:Pochhammer_5}
q^{-r}[m^{(r)}_{ij}-m^{(r+1)}_{ij},m^{(\infty)}_{kl}]=m^{(r)}_{kj}m^{(\infty)}_{il}-m^{(\infty)}_{kj}m^{(r)}_{il}, \,\, i,j,k,l=1,\ldots,n.
\end{equation}
and consequently we get
\begin{equation}\label{eqn:Pochhammer_6}
(1-qu)[m_{ij}(qu),m^{(\infty)}_{kl}]=m_{kj}(u)m^{(\infty)}_{il}-m^{(\infty)}_{kj}m_{il}(u), \,\, i,j,k,l=1,\ldots,n.
\end{equation}
\subsection{Bessel functions.}
In this subsection we give another example based on the family of Bessel functions and deduce defining relations for the corresponding algebra. Properties of this algebra will be studied elsewhere.

Bessel functions $C_{\nu},\nu\in \mathbb{Z}$ satisfy the recurrence relations:
\[
C_{\nu+1}(z)+C_{\nu-1}(z)=\frac{2\nu}{z}C_{\nu}(z),\,\, \nu\in \mathbb{Z}.
\]
Applying the identity \eqref{eqn:GenYangian_0} with $f=C_{\nu+1}(z)+C_{\nu-1}(z)$ and $g=C_{\mu+1}(z)+C_{\mu-1}(z)$, $\nu,\mu\in \mathbb{Z}$ and denoting
\[
c_{ij}^{(\nu)}:=(C_{\nu}(E))_{ij},i,j=1,\ldots,n,\nu\in\mathbb{Z},
\]
\[
B_{ij}(u)=\sum\limits_{\nu\in\mathbb{Z}}c_{ij}^{(\nu)} u^{-\nu},
\]
we get
\[
(\frac{2 v^2}{v^2-1}\partial_v-\frac{2 u^2}{u^2-1}\partial_u)[B_{ij}(v),B_{kl}(u)]=B_{kj}(u)B_{il}(v)-B_{kj}(v)B_{il}(u).
\]

\section{Polarization of the Yangian}\label{sec:Polarization}

\subsection{Polarized ternary relation.}
Let $V$ be an infinite dimensional separable locally convex vector space  in duality $<\cdot,\cdot>$ with $V^{*}$, $\{e_i, \, 1\leq i\}$  be a basis of $V$ and $\{f_i, \, 1\leq i\}$ be the dual basis of $V^{*}$. Denote  $t^{(r)}_{ik}$, $ i,k,r\in\mathbb{N}$ the generators of the Yangian $Y(\gl_{\infty})$ (see p. 295 of the book \cite{Molevbook} for the definition of $Y(\gl_{\infty})$ ).

\begin{defin} 
Let $A\in End(V)$.
Define
\begin{equation}\label{eqn:Gdef}
\widetilde{G}_{ij}(A,u):=\sum\limits_{r=0}^{\infty}\frac{1}{u^r}\sum\limits_{k=1}^{\infty} <Ae_{i},f_{k}> t^{(r)}_{jk},
\end{equation}
\begin{equation}\label{eqn:JSdef}
\widetilde{D}_{ij}(A):=res\, \widetilde{G}_{ij}(A,u)=\sum\limits_{k=1}^{\infty} <Ae_{i},f_{k}> t^{(1)}_{jk},i,j,=1,\ldots,m,u\in\mathbb{C}.
\end{equation}
\end{defin}

From now on we denote $\widetilde{G}(A,u):=\{\widetilde{G}_{ij}(A,u)\}_{i,j\in\mathbb{N}}$, $\widetilde{D}(A):=\{\widetilde{D}_{ij}(A)\}_{i,j\in\mathbb{N}}$, the infinite matrices with operator valued coefficients.

Set
\begin{equation}\label{eqn:JSdef_2}
D(A):= tr\,\,\widetilde{D}(A)=\sum\limits_{i=1}^{\infty}\widetilde{D}_{ii}(A)=\sum\limits_{i,k=1}^{\infty} <Ae_{i},f_{k}> t^{(1)}_{ik}.
\end{equation}

\begin{rem}
Notice that $D$ defined by formula \eqref{eqn:JSdef_2} is the Jordan-Schwinger map as the elements $t_{ij}^{(1)}$  satisfy commutation relations of matrix units $E_{ij}$.
\end{rem}

\begin{prop}[Polarized ternary relation]\label{prop:PTR}
We have
\begin{equation}\label{eqn:PolarizedRTT}
(u-v)[\widetilde{G}_{ij}(A,u),\widetilde{G}_{kl}(B,v)]=\widetilde{G}_{il}(A,u)\widetilde{G}_{kj}(B,v)-\widetilde{G}_{il}(A,v)\widetilde{G}_{kj}(B,u).
\end{equation}
In particular,
\begin{equation}\label{eqn:PolarizedRTT_1}
[\widetilde{D}_{ij}(A),\widetilde{D}_{kl}(B)]=<Ae_{i},f_l>\widetilde{D}_{kj}(B)-<Be_{k},f_j>\widetilde{D}_{il}(A),
\end{equation}
and
\begin{equation}\label{eqn:PolarizedRTT_2}
[D(A),D(B)]=D([B,A]),
\end{equation}
where $A,B\in End(V),i,j,k,l=1,\ldots,u,v\in\mathbb{C}$.
\end{prop}
\begin{proof}
Proof follows from the ternary relation for generators of Yangian $Y(\gl_m)$. Indeed,
\begin{eqnarray}
(u-v)[\widetilde{G}_{ij}(A,u),&\widetilde{G}_{kl}(B,v)] = (u-v)\sum\limits_{r,s=0}^{\infty}\frac{1}{u^r v^s}\sum\limits_{p,l=1}^{\infty}<A e_i,f_l> <B e_k,f_p> [t_{jm}^{(r)},t_{lp}^{(s)}]\nonumber\\
&= \sum\limits_{r,s=0}^{\infty}\frac{1}{u^r v^s}\sum\limits_{p,l=1}^{\infty}<A e_i,f_l> <B e_k,f_p>([t_{jm}^{(r+1)},t_{lp}^{(s)}]-[t_{jm}^{(r)},t_{lp}^{(s+1)}])\nonumber\\
&= \sum\limits_{r,s=0}^{\infty}\frac{1}{u^r v^s}\sum\limits_{p,l=1}^{\infty}<A e_i,f_l> <B e_k,f_p>(t^{(r)}_{lm}t^{(s)}_{jp}-t^{(s)}_{lm}t^{(r)}_{jp})\nonumber\\
&= \widetilde{G}_{il}(A,u)\widetilde{G}_{kj}(B,v)-\widetilde{G}_{il}(A,v)\widetilde{G}_{kj}(B,u).\label{eqn:RTT_proof}
\end{eqnarray}
Thus we have shown the relation \eqref{eqn:PolarizedRTT}. Formula \eqref{eqn:PolarizedRTT_1} follows from \eqref{eqn:PolarizedRTT} as the first order expansion term or can be shown directly from \eqref{eqn:Yangian_p}. Relation \eqref{eqn:PolarizedRTT_2}  follows from \eqref{eqn:Yangian_p}.
\end{proof}
\begin{rem}\label{rem:Polarization}
Consider the bilinear form 
\begin{equation}\label{eqn:Def_Epsilon}
\epsilon (A,B):=tr (A\widetilde{D}(B)-\widetilde{D}(AB)), A,B\in End(V).   
\end{equation}
Then $\epsilon$ is symmetric. This property of 
 $\epsilon$ guarantees that 
   \eqref{eqn:PolarizedRTT_2} follows from \eqref{eqn:PolarizedRTT_1}.
 
\end{rem}

\begin{rem}\label{rem:RTTGeneralization}
It is not always possible to polarize the RTT relation as the following example shows. Let $R$ be a trigonometric $R$-matrix of the six vertex model, i.e.
\[
R(u)=\left(
\begin{array}{cccc}
f(u) & 0 & 0 & 0\\
0 & 1 & g(u) & 0\\
0 & g(u) & 1 & 0\\
0 & 0 & 0 & f(u)
\end{array}
\right),
\]
where $f(u)=\frac{\sinh(u+\eta)}{\sinh u}$, $g(u)=\frac{\sinh\eta}{\sinh u}$, $\eta$ is a parameter. Then the RTT relation has  form
\begin{eqnarray}
&[T^{jk}(u),T^{\alpha\beta}(v)]=g(u-v)\left(T^{\alpha k}(v)T^{j\beta}(u)-T^{\alpha k}(u)T^{j\beta}(v)\right)\nonumber\\
&+((f-g)(u-v)-1)\left(\delta_{k\beta}T^{\alpha k}(v)T^{j\beta}(u)-\delta_{j\alpha}T^{\alpha k}(u)T^{j\beta}(v)\right),
\end{eqnarray}
and we  notice that the term before the last one cannot be polarized.
\end{rem}

\begin{rem}
Relation \eqref{eqn:PolarizedRTT} becomes the ternary relation when $A=B$ is fixed.
\end{rem}

The following statement follows immediately from the formula \eqref{eqn:PolarizedRTT}.
\begin{cor}
\begin{equation}\label{eqn:PolarizedRTTpp}
[\widetilde{G}_{ij}(A,u),\widetilde{G}_{kl}(B,u)]=\frac{d\widetilde{G}_{il}(A,u)}{du}\widetilde{G}_{kj}(B,u)-\widetilde{G}_{il}(A,u)\frac{d\widetilde{G}_{kj}(B,u)}{du},
\end{equation}
$A,B\in\mathcal{L}(V,V),i,j,k,l=1,\ldots,u\in\mathbb{C}$.
\end{cor}

\subsection{Ternary Lie algebra.}

Let $\mathcal{A}$ be a Lie algebra,
 $\psi:\mathcal{A}\to End(V)$  a   representation of $\mathcal{A}$.  Using notation from the previous section we set
\begin{equation}
\widetilde{G}_{ij}(a,u):=\widetilde{G}_{ij}(\psi(a),u),
\end{equation}
\begin{equation}
\widetilde{D}_{ij}(a):=\widetilde{D}_{ij}(\psi(a)),\label{eqn:Tilde_D}
\end{equation}
\begin{equation}
D(a):=D(\psi(a)),i,j=1\ldots,a\in \mathcal{A}.
\end{equation}
Then Proposition~\ref{prop:PTR} gives us 
\begin{equation}\label{eqn:PolarizedRTTa}
(u-v)[\widetilde{G}_{ij}(a,u),\widetilde{G}_{kl}(b,v)]=\widetilde{G}_{il}(a,u)\widetilde{G}_{kj}(b,v)-\widetilde{G}_{il}(a,v)\widetilde{G}_{kj}(b,u).
\end{equation}
In particular,
\begin{equation}\label{eqn:PolarizedRTT_1a}
[\widetilde{D}_{ij}(a),\widetilde{D}_{kl}(b)]=<\psi(a)e_{i},f_l>\widetilde{D}_{kj}(b)-<\psi(b)e_{k},f_j>\widetilde{D}_{il}(a),
\end{equation}
and
\begin{equation}\label{eqn:PolarizedRTT_2a}
[D(a),D(b)]=D([b,a]),
\end{equation}
where
 $a,b\in\mathcal{A},i,j,k,l\in\mathbb{N}$ and $u,v\in\mathbb{C}$.

\begin{defin}\label{def:TernaryAlg}
For a Lie algebra $\mathcal{A}$, an auxiliary linear space $\mathcal{W}$ and a collection of linear maps 
$$\{x_{ij}\in L(\mathcal{A}, End(\mathcal{W})), \,\, i,j\in\mathbb{N}\}$$
define the following subspace of $End(\mathcal{W})$:
\begin{equation}
\mathcal{A}^{tern}(\mathcal{W}, \{x_{ij}\}):=span_{\mathbb C}\{x_{ij}(a), a\in \mathcal{A}, i,j\in\mathbb{N}\}.
\end{equation}
\end{defin}

One can use the formula \eqref{eqn:PolarizedRTT_1a} 
to define the Lie algebra structure on $\mathcal{A}^{tern}(\mathcal{W}, \{x_{ij}\})$. Namely, we have the following proposition which can be easily checked.
\begin{prop} Let $\psi:\mathcal{A}\to End(V)$ be any representation of $\mathcal{A}$.
Then the space $\mathcal{A}^{tern}(\mathcal{W}, \{x_{ij}\})$ is a Lie algebra with the Lie bracket defined as follows:
\begin{equation*}
[x_{ij}(a),x_{kl}(b)]=<\psi(a)e_{i},f_l>x_{kj}(b)-<\psi(b)e_{k},f_j>x_{il}(a),\, i,j,k,l\in\mathbb{N}.
\end{equation*}
\end{prop}

We will denote the resulting Lie algebra by $\mathcal{A}^{tern}(\mathcal{W}, \{x_{ij}\}, \psi)$ and call it the ternary Lie algebra of $\mathcal A$ determined by $\mathcal{W}$, $\{x_{ij}\}$ and $\psi$. We fix now $\mathcal{W}$, $\{x_{ij}\}$, $\psi$ and simply denote the corresponding ternary Lie algebra by $\mathcal{A}^{tern}$. 
Define the bilinear form $\epsilon:\mathcal{A}\times \mathcal{A}\to \mathcal{A}^{tern}$ as follows: 
$$\epsilon(a,b):=\sum\limits_{k,l=1}^{\infty}<\psi(a)e_{k},f_l> x_{lk}(b)-\sum\limits_{m=1}^{\infty} x_{mm}(ab),$$
and the trace map $\tr(x):\mathcal{A}\to End(\mathcal{W})$ as follows:
$$a\mapsto \tr(x)(a):=(\sum\limits_{l=1}^{\infty} x_{ll})(a).$$

\begin{prop}
Assume that the bilinear form $\epsilon$
 is symmetric. Then $-\tr(x)$ is a homomorphism of Lie algebras, that is $\tr(x)(\mathcal{A})$ is a Lie subalgebra of $\mathcal{A}^{tern}$. 
\end{prop} 
 
 If $\tr(x)$ is injective then 
 we have $$\mathcal{A}\cong\tr(\mathcal{A})\subset \mathcal{A}^{tern}.$$

\begin{cor}
    Let $\tilde{D}$ is given by the formula \eqref{eqn:Tilde_D}.  Then the map: $U(\mathcal{A}^{tern})\to Y(\gl_n)$ defined by $$x_{ij}(a)\mapsto \tilde{D}_{ij}(a), \, a\in \mathcal{A}, \,i,j\in\mathbb{N},$$  is a homomorphism. In addition, if $\psi$ is faithful, then this map is an embedding.
\end{cor}

\begin{example}\label{ex:Ternary}
Let $\psi=\ad$, $V=\mathcal{A}$, identified as vector space with $\mathbb{C}^n$ with $<\cdot,\cdot>=(\cdot,\cdot)_{\mathbb{C}^n}$, $\mathcal{W}=\mathbb{C}[x]$.
Put 
\begin{equation*}
x_{ij}(a):=\sum\limits_{k=1}^n <[a,e_{i}],e_{k}> x_j\partial_k,i,j=1,\ldots,n,a\in\mathcal{A}.
\end{equation*}

Using such data we get the following  ternary extensions of algebras of dimensions $2$, $3$ and $4$ (modulo their center):
\begin{trivlist}
\item $A_{2,1}^{tern}=A_{2,1}$.

\item $(A_{2,1}\oplus A_1)^{tern}=A_{3,1}^{tern}=A_{3,3}$.

\item $A_{3,2}^{tern}=A_{3,3}^{tern}=(A_{3,4}^a)^{tern}=(A_{3,5}^b)^{tern}=sl(2)\oplus A_{3,3}$.

\item $sl(2)^{tern}=so(3)^{tern}=gl(3)$.

\item $(sl(2)\oplus A_1)^{tern}=(so(3)\oplus A_1)^{tern}=(A_{4,2}^b)^{tern}=sl(3)\oplus A_{4,5}^{1,1,1}$.
\item $A_{4,1}^{tern}=sl(2)\oplus A_{5,7}^{111}=sl(2)\oplus A_{5,13}^{110}$.
\end{trivlist}
Here we using the nomenclature from \cite{PBNL2003} and \cite{PSWZ76}. In particular, $A_1$ is a $1$-dimensional Lie algebra and $A_{2,1}$ is a $2$-dimensional solvable Lie algebra. 
\end{example}

\subsection{Ternary Yangian}
  Using notation above we introduce the associative algebra 
 $Yt(\mathcal{A})$ as follows:
 \begin{defin}\label{def:TernaryYangian}
For a Lie algebra $\mathcal{A}$, an auxiliary linear space $W$,  representation $\psi:\mathcal{A}\to End(V)$ and a collection of linear maps 
$$\{x_{ij}\in L(\mathcal{A}, End(\mathcal{W})[[u^{-1}]]), \,\, i,j\in\mathbb{N}\}$$
 such that
$$x_{ij}(a)=\psi(a)_{ij}Id_{\mathcal{W}}+\sum\limits_{l=1}^{\infty}X_{ij}^{(l)}(a) u^{-l}, $$
define the following subspace of $End(\mathcal{W})$:
\begin{equation}
Yt(\mathcal{A})(\mathcal{W}, \{x_{ij}\}):=span_{\mathbb C}\{X_{ij}^{(k)}(a), a\in \mathcal{A}, i,j\in\mathbb{N}, k\in \mathbb{N}\}.
\end{equation}
where $X_{ij}^{(k)}\in L(\mathcal{A}, End(\mathcal{W}))$, $k\in \mathbb{N}$, $i,j\in\mathbb{N}$ (we exclude $0$-order coefficient $X_{ij}^{(0)}=\psi(a)_{ij}$).
\end{defin}

One can use the formula \eqref{eqn:PolarizedRTTa} to define the associative algebra structure on $Yt(\mathcal{A})(\mathcal{W}, \{x_{ij}\})$:
\begin{equation}\label{eqn:PolarizedRTTa_X}
(u-v)[x_{ij}(a,u),x_{kl}(b,v)]=x_{il}(a,u)x_{kj}(b,v)-x_{il}(a,v)x_{kj}(b,u),\,i,j,k,l\in\mathbb{N}.
\end{equation}


We will call this algebra \emph{the ternary Yangian} determined by $\mathcal{A}$, $x_{ij}$, $\mathcal{W}$ and $\psi$. When this data is fixed we will simply write $Yt(\mathcal{A})$. 
Since $\mathcal{A}$ is finite dimensional, we can choose a finite number of linearly independent generating functions $x_{ij}(a_k,u)$, $i,j\in\mathbb{N}$, $1\leq k\leq \dim(\mathcal{A})$ of algebra $Yt(\mathcal{A})$.

The ternary Lie algebra $\mathcal{A}^{tern}=\mathcal{A}^{tern}(\mathcal{W}, \{x_{ij}\}, \psi)$ serves as the first order term for $Yt(\mathcal{A})$.
\begin{prop}\label{prop:Evaluation}
  The map  $\rho:Yt(\mathcal{A})\to U(\mathcal{A}^{tern})$ defined by $$x_{ij}(a,u)\mapsto (\psi( a))_{ij}+\frac{1}{u}\widetilde{D}_{ij}(a)$$ is a surjective homomorphism (here the mapping understood as the correspondence between formal series in $u$). 
\end{prop} 
\begin{proof}
    Proof is analogous to the case of Yangian $Y(\gl_n)$. It is enough to notice that the identity \eqref{eqn:PolarizedRTTa_X} is conserved under $\rho$.
\end{proof}
Similarly to the case of $Y(\gl_n)$,  we will call $\rho$ an \textbf{evaluation} epimorphism. 

\begin{thm}\label{thm:deformation}
The algebra $Yt(\mathcal{A})$ is a flat deformation of $\mathcal{U}(\mathcal{A}^{tern}[x])$.
\end{thm}
\begin{proof}
We can rewrite polarized ternary relation as follows
\begin{eqnarray}
[X_{ij}^{(r)}(a),X_{kl}^{(s)}(b)]&=\sum\limits_{m=1}^{ \min(r,s)}(X_{il}^{(m-1)}(a)X_{kj}^{(r+s-m)}(b)-X_{il}^{(r+s-m)}(a)X_{kj}^{(m-1)}(b)),\label{eqn:PolarizedRTT_pt}\\
&i,j,k,l\in\mathbb{N},a,b\in\mathcal{A}, r,s\in\mathbb{N}\cup{0}.\nonumber
\end{eqnarray}
Let us define 
\[
\widetilde{X}_{ij}^{(r)}(a)= h^{r-1}X_{ij}^{(r)}(a),i,j\in\mathbb{N},a\in\mathcal{A}, r\in\mathbb{N}.
\]
Consequently we get
\begin{eqnarray}
[\widetilde{X}_{ij}^{(r)}(a),\widetilde{X}_{kl}^{(s)}(b)] &=
\psi(a)_{il}\widetilde{X}_{kj}^{(r+s-1)}(b)-\psi(b)_{kj}\widetilde{X}_{il}^{(r+s-1)}(a)\nonumber\\
&+h\sum\limits_{m=1}^{\min(r,s)-1}
(\widetilde{X}_{il}^{(m)}(a)\widetilde{X}_{kj}^{(r+s-1-m)}(b)-\widetilde{X}_{il}^{(r+s-1-m)}(a)\widetilde{X}_{kj}^{(m)}(b))\nonumber
\end{eqnarray}
In the limit $h\to 0$ we get the generating relations of $\mathcal{A}^{tern}[x]$. To show flatness let us start with family of linearly independent generators $\{\widetilde{X}_{ij}^{(r)}(a_k),(i,j,k,r)\in [1,\ldots,\dim \mathcal{A}]^3\times \mathbb{N}\}$ of $Yt(\mathcal{A})$. They span $Yt(\mathcal{A})$ by definition.
Let us show that they are algebraically independent. We have algebraic homomorphism
\[
\widetilde{X}_{ij}^{(r)}(a)\mapsto \widetilde{g}_{ij}^{(r)}(a):=h^{r-1}\sum\limits_{m=1}^n \psi(a)_{im} t^{(r)}_{jm}
\]

Assume that $\{\widetilde{X}_{ij}^{(r)}(a_k)\}_{ijkr}$ are not independent, then the family $\{\widetilde{g}_{ij}^{(r)}(a_k)\}_{ijkr}$ is not independent and, consequently, the family $\{t^{(r)}_{ij},i,j\in\mathbb{N},r\in\mathbb{N}\cup{0}\}$ is also algebraically dependent, which is a contradiction. 

\end{proof}
\begin{prop}
$Yt(\mathcal{A})$ is $\mathcal{A}^{tern}$-module with adjoint action.
\end{prop}
\begin{proof}
It is enough to apply formula \eqref{eqn:PolarizedRTT_pt} with $r=1$. Indeed, we have
\begin{equation}
[X_{ij}^{(1)}(a),X_{kl}^{(s)}(b)]=\psi(a)_{il} X_{kj}^{(s)}(b)-X_{il}^{(s)}(a)\psi(b)_{kj},\label{eqn:Moduleaction}
\end{equation}
and the result follows.
\end{proof}
\begin{example}\label{ex:sl2}
Let us show that $Yt(sl(2))=Y(gl(3))$. Indeed, let us denote by $\{e_1,e_2,e_3\}$ the standard basis $sl(2)$ i.e.
\[
[e_1,e_2]=e_3, [e_3,e_2]=-2e_2, [e_3,e_1]=2e_1.
\]
Then we have that
\[
g_{\cdot j}^{(r)}(e_{\cdot})=\left(
\begin{array}{ccc}
0 & t_{j3}^{(r)} & -2t_{j1}^{(r)}\\
-t_{j3}^{(r)} & 0 & 2t_{j2}^{(r)}\\
2t_{j1}^{(r)} & -2t_{j2}^{(r)} & 0
\end{array}
\right),j=1,\ldots,3, r\in \mathbb{N}\cup{0},
\]
and the result follows.
\end{example}
Defining relations \eqref{eqn:PolarizedRTTa_X} can be rewritten in the matrix form. 
Let $X\in End\,(\mathbb{C}^{\infty})\times Yt(\mathcal{A})[[u^{-1}]]$ be an infinite matrix with elements $x_{ij}(a,u)$, $i,j\in\mathbb{N}$.
The following proposition is an analogue of  \cite[Proposition 1.2.2]{Molevbook}.
%
%
\begin{prop}\label{prop:RTTpolar}
\begin{eqnarray}\label{eqn:RTTpolar_1}
&R(u-v)X_1(a,u)X_2(b,v)=X_2(b,v)X_1(a,u)R(u-v)+\frac{(X_2(b,v)X_1(a,u)-X_2(a,v)X_1(b,u))}{u-v}P\nonumber\\
&=X_2(a,v)X_1(b,u)R(u-v)+X_2(b,v)X_1(a,u)-X_2(a,v)X_1(b,u)
,a,b\in\mathcal{A},u,v\in\mathbb{C}\nonumber
\end{eqnarray}
\end{prop}
\section{Acknowledgments}
\noindent V.Futorny is partially supported by NSF of China (12350710178 and 12350710787). 
M. Neklyudov would like to express sincere gratitude to Universidade Federal do Amazonas (UFAM) and Coordenação de Aperfeiçoamento de Pessoal de Nível Superior (CAPES) for their unwavering support.

\section{Data Availability}

We do not analyse or generate any datasets, because our work proceeds within a theoretical and mathematical approach. One can obtain the relevant materials from the references below.

\section{Conflict of Interest}

The authors declare that there is no conflict of interest.

\end{document}